\documentclass[a4,12pt]{amsart}
\usepackage{amssymb,amsfonts,amsmath,bbm,enumerate,color,url}
\usepackage{amsthm}	
\usepackage{faktor}

\usepackage{mathbbol}

\usepackage[utf8]{inputenc}

\def\:{\thinspace:\thinspace}
\newtheorem{theo}{Theorem}
\newtheorem{lemma}[theo]{Lemma}
\newtheorem{prop}[theo]{Proposition}
\newtheorem{cor}[theo]{Corollary}
\newtheorem{defi}[theo]{Definition}

\newtheorem{rem}[theo]{Remark}

\numberwithin{equation}{section}
\numberwithin{theo}{section}
\def\:{\thinspace:\thinspace}

\theoremstyle{definition}

\numberwithin{theo}{section}
 \def\mP{\mathsf{P}}

\newcommand{\ud}{\,\mathrm{d}}
\newcommand{\e}{\,\mathrm{e}}

\DeclareMathOperator{\dist}{dist}

\newcommand{\Graph}{\mathcal{G}}						
\newcommand{\R}{\mathbb{R}}						
\newcommand{\N}{\mathbb{N}}						
\newcommand{\K}{\mathbb{K}}						
\newcommand{\one}{\mathbbm{1}}

\title{Ornstein--Uhlenbeck Semigroups on Star Graphs}

\author{Delio Mugnolo}
\address{Delio Mugnolo, Lehrgebiet Analysis, Fakult\"at Mathematik und Informatik, Fern\-Universit\"at in Hagen, D-58084 Hagen, Germany}
\email{delio.mugnolo@fernuni-hagen.de}

\author{Abdelaziz Rhandi}
\address{Abdelaziz Rhandi, Dipartimento di Matematica, Università degli Studi di Salerno, Via Giovanni Paolo II, 132, 84084 Fisciano (SA), Italy}
\email{arhandi@unisa.it}

\thanks{The work of D.M.\ was supported by the Deutsche Forschungsgemeinschaft (Grant 397230547). \\
AR is member of G.N.A.M.P.A. of the Italian Istituto Nazionale di Alta Matematica (INdAM).\\
This work was started while AR visited the University of Hagen. He wishes
to express his gratitude to the University of Hagen for the financial support.\\
This article is based upon work from COST Action 18232 MAT-DYN-NET, supported by COST (European Cooperation in Science and Technology), \url{www.cost.eu}.}


\dedicatory{Dedicated to the memory of Rosa Maria Mininni}

\begin{document} 
\maketitle
\begin{abstract}
We prove first existence of a classical solution to a class of parabolic problems with unbounded coefficients on metric star graphs subject to Kirchhoff-type conditions. The result is applied to the Ornstein--Uhlenbeck and the harmonic oscillator operators on metric star graphs. We give an explicit formula for the associated Ornstein--Uhlenbeck semigroup and give the unique associated invariant measure. We show that this semigroup inherit the regularity properties of the classical Ornstein--Uhlenbeck semigroup on $\R$. 
\end{abstract}


\maketitle
\section{Introduction}
The aim of this note is to present some preliminary results in the study of elliptic operators with unbounded coefficients on non-compact metric graphs. For the sake of simplicity, we here restrict first and foremost to the case of graphs with the simplest possible topology, i.e., metric star graphs $\mathcal S_m$ consisting of $m<\infty$ halflines; and to the best understood class of operators with unbounded coefficients, viz the Ornstein--Uhlenbeck operators: these are, on $\R$, the operators associated with the Ornstein--Uhlenbeck stochastic process, i.e.,
they are defined by
\begin{equation}\label{eq:opou-R}
 A f(x)= \frac{1}{2}f''(x) -xf'(x), \qquad x\in \R.
\end{equation}

The theory of second-order differential operators on \textit{compact} metric graphs $\mathcal G$ is classical and goes back to Lumer~\cite{Lum80,Lum80b} and Pavlov--Faddeev~\cite{PavFad83}. Shortly afterwards, Roth \cite{Rot84} presented an explicit formula for the heat kernel -- the integral kernel of the semigroup generated by the plain Laplacian with natural (i.e., continuity and Kirchhoff-type) boundary conditions; this was later extended to more general vertex conditions~\cite{Nic85}, to (possibly) infinite equilateral graphs~\cite{Cat99}, and recently to (possibly) infinite graphs of bounded geometry \cite{BecGreMug21}, to which the case of $\mathcal S_m$ can be reduced by elementary arguments. It is known that, just like its counterpart on $\R$, this semigroup can be associated with the Brownian motion on $\Graph$~\cite{KosPotSch12}.
 Qualitative properties of differential operators of order two \cite{HusMug20} and three \cite{MugNojSei17} on $\mathcal S_m$ have been recently studied, too; metric graphs including semi-infinite intervals appear in the study of linear scattering theory (\cite[Section~5.4]{BerKuc13} and references therein) and nonlinear Schrödinger equations, ever since~\cite{Noj14}.
In virtually all of these investigations, the relevant functional setting is the Hilbert space $L^2(\mathcal G)$ with respect to the measure on $\Graph$ canonically induced by the Lebesgue measure on each halfline $\R_+$. As usual in the theory of  operators with unbounded coefficients, we will instead introduce an appropriate measure adapted to our setting; this will turn out to be the invariant measure for the Ornstein--Uhlenbeck semigroup.

Let us now describe our main results and the structure of the paper. In Section~\ref{sec:gen-setting} we will recall the basic notions in the theory of metric graphs and introduce the relevant function spaces. In Section~\ref{sec:operators} we will introduce the class of operators we are going to study and prove that they drive well-posed evolution equations on $\mathcal S_m$. In a certain sense, our approach here is similar to that of~\cite{BecGreMug21}: our proofs are based on a kind of transference principle, as we extend the properties of the semigroup's integral kernel on $\R$ to explicitly define the integral kernel of the semigroup on $\mathcal S_m$, thus proving existence of a classical solution for a certain class of initial data.
In this article, this is done by making good use of the symmetries of the metric star graph and invariance properties of the semigroup on $\R$. This strategy can probably be pursued in greater generality, as long as the existence of an integral kernel for the relevant semigroup is known.

 In Section~\ref{sec:functional} we turn to the issue of studying the operator theoretical properties of the semigroup associated with this integral kernel. To this purpose, we focus on an especially interesting special case
and study its realizations in the space of bounded continuous functions as well as on Lebesgue spaces, either with respect to the Lebesgue measure or to a suitable alternative measure -- the \textit{invariant measure} associated with the Ornstein--Uhlenbeck process on $\mathcal S_m$. In this way, we are especially able to prove the existence of a consistent family of analytic, positive, compact Ornstein--Uhlenbeck semigroups on $L^p_\mu$-spaces; we can also determine their spectra.
 
Finally, as an application of our results,  in Section~\ref{sec:harmonic} we briefly discuss the behavior of the harmonic oscillator on a metric star graph; indeed, it is already well-known that on $\R$ the relevant Hamiltonian is similar to the Ornstein--Uhlenbeck operator on $L^2_\mu(\R)$, and in particular they have equal spectrum. To our knowledge, the properties of this physical model on graphs have never been studied in the literature; although we wish to mention a well-known model of irreversible quantum graphs due to Smilansky and Solomyak that boils down to coupling a Laplacian on ${\mathcal G}$ with a harmonic oscillator on $\R$~\cite{Smi03,Sol03}, thus defining an operator on $L^2(\Graph)\oplus L^2(\R)\simeq L^2(\Graph\times \R)$.

\section{General setting}\label{sec:gen-setting}

Object of our investigations here is a metric star graph, ${\mathcal S}_m$, with $m$ rays of semi-infinite length, $m \in \N$; i.e., ${\mathcal S}_m$ is the quotient space 
\[
\bigsqcup_{i=1}^m \faktor{[0,\infty)}{\sim}
\] 
that consists of $m$ disjoint half-lines $[0,\infty)$ whose origins are identified with one common zero point, $0$.

Here $\sim$ is the equivalence relation: $(x,i)\sim (y,j)$ if and only if either $x=y$ and $i=j$, or else $x=0$ and $y=0$, regardless of $i,j$.  Accordingly, we use the notations $0 := (0,i)$ for all $i$, as well as $x_i := (x,i)$ and $|x_i| := x$ whenever $x>0$.
We refer to~\cite{Mug19} for a more detailed description of this formalism that, in particular, allows us to extend to a metric graph any metric-measure structure supported on individual (semi-infinite) intervals.

We denote by ${\mathcal S}_m^n$, $n > 0$ the \emph{truncated star} defined by 
\[
{\mathcal S}_m^n:=\bigsqcup_{i=1}^m \faktor{[0,n]}{\sim},
\]
with the equivalence relation $\sim$ defined as above.

It is well known that $\mathcal S_m$ is a metric-measure space with respect to the metric-measure structure induced edgewise by the Euclidean distance and the Lebesgue measure. As already pointed out in the introduction, along with the Lebesgue measure there is another canonical measure that endows $L^p$-spaces when studying the Ornstein--Uhlenbeck operator; it is given by 
\begin{equation} \mu(\ud x)=\frac{2}{m\sqrt{\pi}}\e^{-|x|^2}\ud x.
\end{equation} 
Because $\ud \mu$ is absolutely continuous with respect to the Lebesgue measure $\ud x$, one sees that $\ud \mu$ is locally finite with respect to the Euclidean distance, too: we conclude that $\Graph$ is a metric-measure space with respect to the path metric and the direct sum measure induced by the measure $\ud \mu$. Clearly, also $\mathcal S^n_m$ are metric-measure spaces. In particular, this allows us to consider, without ambiguity, functions spaces based on topological and measure-theoretical notions: in particular, the spaces $C_b(\mathcal S_m)$ (resp., $BUC(\mathcal S_m)$) of bounded (resp., bounded uniformly continuous) functions on $\mathcal S_m$; and the Lebesgue spaces $L^p(\mathcal S_m)$
(resp., $L^p_\mu(\mathcal S_m)$) with respect to the Lebesgue measure (resp., to the measure $\mu$). Likewise, one defines the Sobolev space $W^{1,p}(\mathcal S_m)$ (resp., $W^{1,p}_\mu(\mathcal S_m)$) as the space of functions in $C(\mathcal S_m)\cap L^p(\mathcal S_m)$ (resp., $C(\mathcal S_m)\cap L^p_\mu(\mathcal S_m)$) that are weakly differentiable with a weak derivative in $L^p(\mathcal S_m)$ (resp., $L^p_\mu(\mathcal S_m)$). By definition of $\mathcal S_m$ as a disjoint union, any function $f:\mathcal S_m\to \K$ can be equivalently regarded as a family $(f_i)_{1\le i\le m}$, where $f_i:\R_+\to \K$.

In particular, if $f\in C(\mathcal S_m)$, then in agreement with the above convention we write $f(0):=\lim\limits_{x \to 0}f_i(x)$, $1\le i\le m$.

\section{Operators with unbounded coefficients on metric star graphs}\label{sec:operators}

We want to study first the Kolmogorov operator
\begin{equation}\label{eq:def-L} 
L f(x_i)= q(|x_i|)f''(x_i) +b(|x_i|)f'(x_i)+c(|x_i|)f(x_i), \qquad |x_i|\ge 0, \quad  i=1,\ldots, m, 
\end{equation}                
on $C_b({\mathcal S}_m)$, where $q,b,c \in C^\alpha_{\hbox{\scriptsize loc}}([0,\infty))$ for some $\alpha \in (0,1)$, $b(0)=0$, $q(x) > 0$ for all $x \in [0,\infty)$ and $\sup c\le c_0$ for some $c_0\in \R$. We equip it with continuity along with Kirchhoff-type condition in zero by defining it on the domain
\begin{equation*} D(L)=\{f \in C_b({\mathcal S}_m) \cap  \bigcap_{1\leq p < \infty} \widetilde{W^{2,p}_{\hbox{\scriptsize loc}}}({\mathcal S}_m):  \sum_{i=1}^m f'(0_i)=0 \hbox{ and } Lf \in C_b({\mathcal S}_m)\},
\end{equation*}
where
\begin{equation}\label{eq:sobol-defin1}
\widetilde{W^{k,p}_{\hbox{\scriptsize loc}}}({\mathcal S}_m):=\bigoplus_{i=1}^m W^{k,p}_{\hbox{\scriptsize loc}}(\R_+), \qquad k\in \mathbb N.
\end{equation}
 (Note that, unlike for $W^{1,p}({\mathcal S}_m)$ defined in Section~\ref{sec:gen-setting}, we are \textit{not} imposing continuity at 0 on the functions in $\widetilde{W^{k,p}_{\hbox{\scriptsize loc}}}({\mathcal S}_m)$.)
We associate with the operator $L$ a further operator $\tilde{L}$, acting on the function space $C_b(\R)$, defined by 
\begin{equation*} \tilde{L} f(x) = \tilde{q}(x)f''(x) + \tilde{b}(x)f'(x)+\tilde{c}(x)f(x) \end{equation*}
with domain 
\begin{equation*} D(\tilde{L})=\{f \in C_b(\R) \cap  \bigcap_{1\leq p < \infty} W^{2,p}_{\hbox{\scriptsize loc}}(\R): \tilde{L}f \in C_b(\R)\}, 
\end{equation*}
where 
\begin{eqnarray}\label{eq:extention-q}
& & \tilde{q}(x)=q(x),\,\tilde{b}(x)=b(x),\,\tilde{c}(x)=c(x) \hbox{\ if } x \geq 0\,\hbox{\ and } \nonumber \\
& & \tilde{q}(x)=q(-x), \, \tilde{b}(x)=-b(-x),\,\tilde{c}(x)=c(-x) \hbox{\ if } x \leq 0.
\end{eqnarray} 

In this section we are mainly interested in the existence and uniqueness of solutions to the parabolic problem
\begin{equation}\tag{${\rm P}_\Lambda$} 
\left\{ 
\begin{split}
\partial_t u(t,\cdot)  &= \Lambda u(t,\cdot), \quad t > 0, \\
u(0,\cdot)  & = f(\cdot), 
\end{split}
 \right.\end{equation}
where the subscript in $(P_\Lambda)$ always indicates which operator $\Lambda$ is currently under consideration.

The following remark is crucial for our study.
\begin{rem}
We observe that every function $f \in C_b({\mathcal S}_m)$ uniquely determines $m$ functions $\tilde{f}_i \in C_b(\R)$ given by 
\begin{equation} \label{R}
\tilde{f}_i(x):=\left \{ \begin{array}{lll} f(x_i), & |x_i|=x, & \hbox{if } x \geq 0, \\ \frac{2}{m}\sum\limits_{\substack{1\le j\le m}} f(-x_j)- f(-x_i), & |x_i|=-x, & \hbox{if } x \leq 0, \end{array} \right.
\quad i=1,\ldots,m.
\end{equation}
\end{rem}
Classical solutions to $({\rm P}_L)$ for $L$ as defined in \eqref{eq:def-L}, are defined as follows.
\begin{defi}
A function $u  \in C_b([0,\infty) \times {\mathcal S}_m)$ is called \emph{classical solution} of $({\rm P}_L)$ if $u(\cdot,x) \in  C^1((0,\infty))$ for every $x \in {\mathcal S}_m$, $u(t,\cdot) \in D(L)$ for every $t>0$ and $u$ satisfies $({\rm P}_L)$.
\end{defi}

The main result of this section concerns existence of solution to the problem $({\rm P}_L)$.
\begin{theo}\label{main-thm}
Assume that $q,\,b,\,c$ are in $C^\alpha_{\hbox{\scriptsize loc}}([0,\infty))$ for some $\alpha \in (0,1)$, $q(x) > 0$ for all $x \in [0,\infty)$, $\sup c\le c_0$ for some $c_0\in \R$ and $b(0)=0$. Then, 
for every function $f \in C_b({\mathcal S}_m)$, there exists at least one classical solution of $({\rm P}_L)$. 

Furthermore, if the solution to $({\rm P}_{\tilde{L}})$ is unique then so is the solution of $({\rm P}_{{L}})$. In that case the semigroup $(T_m(t))_{t\ge 0}$ generated by $L$ on 
$C_b({\mathcal S}_m)$ is given by
\begin{equation} \label{semiT} 
\begin{split}
T_m(t)f(x_i)&=  \int_{(\R_+,i)} \big (k(t,|x_i|,|y_i|)- k(t,|x_i|,-|y_i|) \big )f(y_i)\, \ud y_i \\
      &\quad +  \sum_{ j=1}^m \int_{(\R_+,j)} \frac{2}{m}k(t,|x_i|,-|y_j|)f(y_j)\, \ud y_j,\quad f\in C_b({\mathcal S}_m),\,x_i\in {\mathcal S}_m,\,i=1,\ldots ,m, 
      \end{split}
\end{equation}
where $k$ is the integral kernel of the semigroup $(T(t))_{t\ge 0}$ generated by $\tilde{L}$. Moreover if $c\equiv 0$ and $T(t_0)\one =\one$ for some $t_0>0$, then $T_m(t)\one=\one$ for all $t\ge 0$ (i.e., $T_m(\cdot)$ is
\emph{conservative}).
\end{theo}        

At the danger of being redundant, we stress that the integral kernel $k$ depends on $q$, $b$, and $c$. Also, we do not expect $T_m(\cdot)$ to be strongly continuous.


\begin{proof}
To construct a solution for the initial data $f\in C_b({\mathcal S}_m)$ we first consider problem $({\rm P}_L)$ on the truncated stars ${\mathcal S}_m^n$, $n \in \N$, with initial data $f_{|{\mathcal S}_m^n}$ and with Dirichlet boundary conditions on the endpoints $(n,i)$ for each $i=1,\ldots,m$.
For each $n\in \N$ and $i=1,\ldots ,m$ we consider the Cauchy--Dirichlet problem
\begin{equation}\label{eq:troncated}
\left\{ 
\begin{aligned} 
\partial_t u_i^n(t,\cdot)  &= \tilde{L} u_i^n(t,\cdot), && t > 0, \\
u_i^n(t,-n) &= u_i^n(t,n)=0, && t > 0,\\
u_i^n(0,x)  &= \tilde{f}_i(x), && x\in (-n,n), 
\end{aligned} 
\right. 
\end{equation}
where $\tilde{f}_i$ is the function given by \eqref{R}. By classical results, cf. \cite[Theorem 9.4.1]{Lorenzi-Rhandi},
for parabolic Cauchy problems in bounded
domains we know that the above problem
has a unique solution $u_i^n \in C([0,\infty)\times (-n,n))\cap C^{1+\frac{\alpha}{2},2+\alpha}_{\scriptsize loc}((0,\infty)\times [-n,n])$, $i=1,\ldots,m$.

We now define a function $\hat{u}^n$ on $[0,\infty)\times {\mathcal S}_m^n$ by 
\begin{equation}\label{func:graph}
\hat{u}^n(t,x_i):=u_i^n(t,|x_i|),\quad i=1,\ldots,m,\quad |x_i| \ge 0,\quad t \geq 0.
\end{equation}
In order to prove that $\hat{u}^n$ is a ``classical'' solution of problem (P$_{L}$) 
on ${\mathcal S}_m^n$, we only have to verify that $\hat{u}^n(t,\cdot)\in D(L_{|_{C_b(S^n_m)}})$ for all $t>0$, i.e.,
\begin{equation}\label{continuity}
\hat{u}^n(t,0_i)=\hat{u}^n(t,0_j)   \hbox{ for all } i,j \in \{1,\ldots,m\} \quad \hbox{and} \quad \sum_{i=1}^m (\hat{u}^n)'(t,0_i)=0,  \quad t > 0.
\end{equation}
Given the functions $\tilde{f}_i$ constructed according to \eqref{R}, $i=1,\ldots,m$, let us define functions $\tilde{f}_{i,j}:\R\to\R$ by $\tilde{f}_{i,j}(x):=\tilde{f}_i(x)-\tilde{f}_j(-x)$ for each $i,j=1,\ldots,m$. Now, each such $\tilde{f}_{i,j}$ is odd, since by construction $\tilde{f}_{i,j}=\tilde{f}_{j,i}$.
 Therefore, using the definition of the functions $\tilde{q}$ and $\tilde{b}$ we deduce that the unique solution $v_{ij}^n(t,x):=u_i^n(t,x)-u_j^n(t,-x)$ of \eqref{eq:troncated} with initial data $\tilde{f}_{i,j}$ is odd
 and especially $v_{ij}^n(t,0)=0$ for all $t \geq 0$. This and \eqref{func:graph} imply the continuity of $\hat{u}^n$, i.e., the first condition in~\eqref{continuity}. 
To prove the second condition in~\eqref{continuity} one considers the solution $v^n(t,x)=\sum_{i=1}^mu_i^n(t,x)$ of \eqref{eq:troncated} with initial data the function $\tilde{f}(x)=\sum_{i=1}^m \tilde{f}_i(x)$, which is even. Thus, again by \eqref{eq:extention-q}, one deduces that $v^n(t,x)=v^n(t,-x)$. This proves the second condition in~\eqref{continuity}.


Now, using Schauder interior estimates and a compactness argument, cf.~\cite[Theorem 2.2.1]{LorBer07}, we know that for
each $i=1,\ldots,m$, the function $u_i:[0,\infty)\times \R\to \R$
$$u_i(t,x):=\lim_{n\to \infty}u_i^n(t,x)$$ exists for any $t\ge 0$ and any $x\in \R$, and belongs to 
$C([0,\infty)\times \R)\cap C^{1+\frac{\alpha}{2},2+\alpha}_{\hbox{ \scriptsize loc}}((0,\infty)\times \R)$. Moreover for each $i=1,\ldots,m,\,u_i$ is a classical solution of 
(P$_{\tilde L}$) on $(0,\infty) \times \R$ with initial data $f_i$ and satisfies
\begin{equation}\label{contract.}
|u_i(t,x)|\le \e^{c_0 t}\|\tilde{f}_i\|_\infty ,\quad t>0,\,x\in \R ,\,i=1,\ldots ,m.
\end{equation}
If we set $T(t)\tilde{f}_i(x):=u_i(t,x)$, then $(T(t))_{t\ge 0}$ is a semigroup on $C_b(\R)$ satisfying
$$T(t)f(x)=\int_{\R}k(t,x,y)f(y)\,dy,\quad f\in C_b(\R),\,t>0,\,x\in \R ,$$
where the kernel $k$ is strictly positive, $k(t,\cdot ,\cdot)$ and $k(t,x,\cdot)$
are measurable for any $t>0,\,x\in \R$ and for a.e. fixed $y\in \R$, $k(\cdot ,\cdot ,y)\in C^{1+\frac{\alpha}{2},2+\alpha}_{\scriptsize loc}((0,\infty)\times \R)$ and it is a solution of $\partial_t u-\tilde{L}u=0$, cf. \cite[Theorem 2.2.5]{LorBer07}.

Defining now the function $T_m(t)f: {\mathcal S}_m\to \R$ by
\[
T_m(t)f(x_i):=T(t)\tilde{f}_i(|x_i|),\qquad i=1,\ldots,m,\ |x_i| \ge 0,\ t \geq 0,
\]
 and using \eqref{continuity}, we arrive at the desired classical solution of (P$_L$) on ${\mathcal S}_m$. Moreover, using \eqref{contract.}, we deduce that $(T_m(t))_{t\ge 0}$ is a semigroup of contractions on $C_b({\mathcal S}_m)$. On the other hand, by \eqref{R}, we have 
\[
\begin{split} T_m(t)f(x_i)&=  \int_{\R_+} k(t,|x_i|,y)f(y)\, \ud y + \int_{\R_-} k(t,|x_i|,y)f(y)\, \ud y \\
  &=  \int_{(\R_+,i)} \big (k(t,|x_i|,|y_i|)- k(t,|x_i|,-|y_i|) \big )f(y_i)\, \ud y_i \\
      & \qquad +  \sum_{ j=1}^m \int_{(\R_+,j)} \frac{2}{m}k(t,|x_i|,-|y_j|)f(y_j)\, \ud y_j. \notag
\end{split}
\]
If furthermore $c\equiv 0$ and $T(t_0)\one =\one$ for some $t_0>0$, then, by \cite[Proposition 4.1.10]{LorBer07},  $T(t)\one =\one$ for  all $t\ge 0$, and hence $\int_{\R}k(t,x,y)\,dy=1$ for  all $t>0,\,x\in \R$. This implies that
$$ \int_{(\R_+,i)} k(t,|x_i|,|y_i|)\,\ud y_i+\int_{(\R_+,j)} k(t,|x_i|,-|y_j|)\,\ud y_j=1,\quad \forall i,\,j=1,\ldots ,m,\,x\in {\mathcal S}_m,\,t>0,$$ holds. So, it follows from \eqref{semiT} that the  semigroup  $(T_m(t))_{t\ge 0}$ is conservative.

Finally, if $u$ is a solution of (P$_L$) with initial data $f\equiv 0$, then for each $i=1,\ldots ,m$, $u_i$ defined by \eqref{R} is a solution of (P$_{\tilde L}$) with $u_i(0,\cdot)=0$. Thus, the uniqueness of the solution to (P$_{\tilde L}$) implies that $u_i\equiv 0$ for each $i=1,\ldots ,m$, and hence $u\equiv 0$.
\end{proof} 
 
\begin{rem}\label{rem:3-4}
\begin{itemize}
\item[(a)] The formula \eqref{semiT} shows that the semigroup $(T_m(t))_{t\ge 0}$ is positive provided
$$k(t,|x_i|,|y_i|)\ge k(t,|x_i|,-|y_i|),\quad t>0,\,x_i,\,y_i\in {\mathcal S}_m.$$ This is especially the case for the Ornstein--Uhlenbeck kernel
$$k_{OU}(t,x,y):=\frac{1}{\sqrt{\pi(1-\e^{-2t})}}\exp[-(1-\e^{-2t})^{-1}(\e^{-t}x-y)^2],\quad t>0,\,x,y\in \R .$$
\item[(b)]
The representation \eqref{semiT} allows us to extend the semigroup to the space $B({\mathcal S}_m)$ of bounded and measurable functions. Moreover this semigroup has the strong Feller property, since $T_m(t)f(x_i)=T(t)\tilde{f}_i(|x_i|)$ and $T(\cdot)$ has the strong Feller property, cf. \cite[Proposition 2.2.12]{LorBer07}.
\end{itemize}
\end{rem}

\section{The Ornstein--Uhlenbeck semigroup on metric star graphs}\label{sec:functional}

As already mentioned in the introduction, a particularly interesting special case of the class operators studied above
is the Ornstein--Uhlenbeck type operator
\begin{equation}\label{eq:opou}
 A f(x_i)= \frac{1}{2}f''(x_i) -|x_i|f'(x_i), \qquad |x_i|\ge 0, \quad  i=1,\ldots, m, 
\end{equation}
 with Kirchhoff-type condition in zero encoded in the domain 
\begin{equation}\label{eq:domou}
D(A)=\{f \in C_b({\mathcal S}_m) \cap  \bigcap_{1\leq p < \infty} W^{2,p}_{\hbox{ \scriptsize loc}}({\mathcal S}_m): \sum_{i=1}^m f'(0_i)=0 \hbox{ and } Af \in C_b({\mathcal S}_m)\}. 
\end{equation}
For $m=1$ we have the Ornstein--Uhlenbeck operator on $\R^+$ with Neumann boundary condition in zero and for $m=2$ the Ornstein--Uhlenbeck operator on $\R$. Therefore our setting can be regarded as a generalization of these well known cases.

As a consequence of Theorem \ref{main-thm} and Remark \ref{rem:3-4} we have the following, where we denote by $S(\cdot)$ the classical Ornstein--Uhlenbeck semigroup on $C_b(\R)$.

\begin{prop}
For every $f \in C_b({\mathcal S}_m)$ there exists a unique bounded, classical solution $u$ of $({\rm P}_A)$. This solution is given by the so-called \emph{Ornstein--Uhlenbeck semigroup on ${\mathcal S}_m$}
\begin{equation}\label{semiS} 
\begin{split}
(S_m(t)f)(x_i)&:=u(t,x_i)=S(t)\tilde{f}_i(|x_i|)\\
&= \frac{1}{\sqrt{\pi(1-\e^{-2t})}}\int_{(\R_+,i)} \Big ( \exp[-(1-\e^{-2t})^{-1}(\e^{-t}|x_i|-|y_i|)^2] \\
&\qquad\quad  - \exp[-(1-\e^{-2t})^{-1}(\e^{-t}|x_i|+|y_i|)^2] \Big ) f(y_i)\,\ud y_i \\
&\qquad  + \frac{2}{m\sqrt{\pi(1-\e^{-2t})}}\sum_{\substack{1\le j\le m}}
 \int_{(\R_+,j)} \exp[-(1-\e^{-2t})^{-1}(\e^{-t}|x_i|+|y_j|)^2]f(y_j)\,\ud y_j
\end{split}
\end{equation}
for $1\le i\le m$.
Furthermore, $S_m(\cdot)$ is irreducible, conservative, contractive and has the strong Feller property.  
\end{prop} 
\begin{proof}
It suffices to prove that $S_m(\cdot)$ is irreducible and contractive. To prove the contractivity of $S_m(\cdot)$ we take $f\in C_b({\mathcal S}_m)$ and $t\ge 0$. Now, by Remark \ref{rem:3-4}.(a), $S_m(\cdot)$ is positive and so, $S_m(t)\one=\one$ implies that
$$|S_m(t)f|\le S_m(t)|f|\le \|f\|_\infty S_m(t)\one =\|f\|_\infty.$$

To show the irreducibility
 of $S_m(\cdot)$, let us consider $0\le f\in C_b({\mathcal S}_m)$ such that $f\not\equiv 0$. Assume, by contradiction, that there is $x=x_i\in {\mathcal S}_m$ and $t>0$ such that
$S_m(t)f(x_i)=0$. So, by \eqref{semiS} and Remark \ref{rem:3-4}.(a), we have
\begin{equation}\label{eq:irreduc}
 \int_{(\R_+,j)} \exp[-(1-\e^{-2t})^{-1}(\e^{-t}|x_i|+|y_j|)^2]f(y_j)\,\ud y_j=0,\quad \forall j=1,\ldots m.
 \end{equation}
Thus, $f\equiv 0$, which is a contradiction.
\end{proof}

\begin{rem}
In view of~\eqref{semiS}, an equivalent formula for the Ornstein--Uhlenbeck semigroups $(S_m(t))_{t\ge 0}$ is as follows:
\[
\begin{split}
(S_m(t)f)(x_i)
&= \frac{1}{\sqrt{\pi(1-\e^{-2t})}}\int_{(\R_+,i)}  \e^{-\frac{(\e^{-t}|x_i|-|y_i|)^2}{(1-\e^{-2t})}} f(y_i)\,\ud y_i\\
&\qquad  + \frac{1}{\sqrt{\pi(1-\e^{-2t})}}\sum_{\substack{1\le j\le m}} \int_{(\R_+,j)} \sigma_{ij}\e^{-\frac{(\e^{-t}|x_i|+|y_j|)^2}{(1-\e^{-2t})}} f(y_j)\,\ud y_j,
\end{split}
\]
where $\Sigma=(\sigma_{ij})$ is the scattering matrix defined by
\[
\sigma_{ij}:=\left\{
\begin{aligned}
&\frac{2-m}{m},\qquad &&\hbox{if }i=j,\\
&\frac{2}{m}, &&\hbox{otherwise}.
\end{aligned}
\right.
\]
In other words, the integral kernel of $(S_m(t))_{t\ge 0}$ can be obtained as the overlapping of the plain Ornstein--Uhlenbeck kernel on $\mathbb R$ (corresponding to the unscattered movement of a particle between two points of the same edge of ${\mathcal S}_m$) and the weighted sum of the paths between $x$ and a point $y$ on different edges (with weight $\frac{2}{m}$), or else on the same edge with $y$ reached only after the particle has been reflected in the center of the star (with weight $\frac{2-m}{m}$). Observe that if $m=2$, then no reflection is possible and the above formula yields just the usual Ornstein--Uhlenbeck semigroup $S(\cdot)$ on $\mathbb R$. Following the ideas of~\cite{Rot84} we can conjecture that this formula may be generalized to an arbitrary graph ${\mathcal G}$ as follows:
\[
(\widetilde{S}(t)f)(x)=\int_{\mathcal G} p(t,x,y)f(y)dy\qquad \hbox{for all }t>0 \hbox{ and }x\in {\mathcal G},
\]
where, for any two $x,y\in {\mathcal G}$, ${\mathfrak P}_{x,y}$ denotes the set of all paths from $x,y$, 
\[
p(t,x,y):=\sum_{\mP \in {\mathfrak P}_{x,y}} \sigma(\mP) G_1(t,\dist(\e^{-t}x,y)),\qquad t>0,\; x,y\in {\mathcal G},
\]
$G_1$ is the convolution kernel of the usual Ornstein--Uhlenbeck semigroup on $\mathbb R$, i.e.
\[
G_1(t,z):=\frac{\e^{-\frac{|z|^2}{1-\e^{-2t}}}}{\sqrt{\pi(1-\e^{-2t})}},\qquad t>0,\ z\in \mathbb R,
\]
$\sigma(\mP)$ is the product of all scattering coefficients along the path $\mP$, and $\dist(x,y)$ is the distance between $x,y$ on ${\mathcal G}$ viewed as a metric-measure space. This is in the spirit of the generalized Roth formulae for heat kernels discussed in~\cite{BecGreMug21}, albeit our metric star graphs do not formally satisfy the assumptions on boundedness of the graph's geometry therein.
\end{rem}

\begin{rem}\label{rem:Cb}
The usual properties of the Ornstein--Uhlenbeck semigroup, cf. \cite[Sections 9.2 and 9.4]{LorBer07}, hold (with the same proofs):
\begin{itemize}
\item Since $S_m(t)f(x_i)=S(t)\tilde{f}_i(|x_i|)$ and $C_0(\R)$ is invariant for $S(t)$, it follows that $S_m(t)$ maps $C_0({\mathcal S}_m)$ into $C_0({\mathcal S}_m)$ for all $t\ge 0$.
\item $S_m(t)$ is not compact on $C_b({\mathcal S}_m)$. This can be proven as in \cite[Theorem 5.1.11]{LorBer07}.
\item $S_m(\cdot)$ is not strongly continuous on $C_b({\mathcal S}_m)$. More specifically, $\lim_{t \to 0} \|S_m(t)f-f\|_\infty= 0$ iff  $f \in BUC({\mathcal S}_m)$ and $\lim_{t \to 0} |f(\e^{-t}x_i)-f(x_i)|= 0$ uniformly with respect to $x_i \in {\mathcal S}_m$. 
\item 
From \eqref{semiS} one deduces that $S_m(\cdot)$ extrapolates to a consistent family of strongly continuous semigroups on $L^p({\mathcal S}_m)$ for all $1 \leq p < \infty$.
\end{itemize}
\end{rem} 
The following result gives the unique invariant measure of $S_m(\cdot)$.           
\begin{theo}
There exists a unique invariant probability measure $\mu_m$ for the Ornstein--Uhlenbeck semigroup $S_m(\cdot)$. This measure has density 
\begin{equation} \mu_m(\ud x_i)=\frac{2}{m\sqrt{\pi}}\e^{-|x_i|^2}\ud x_i, \quad i=1, \ldots,m, \end{equation} 
with respect to Lebesgue measure.  
\end{theo}
\begin{proof}
Let $f \in C_b({\mathcal S}_m)$, $\tilde{f}_i \in C_b(\R)$ as in \eqref{R}, $i=1,\ldots,m$. Let $S(\cdot)$ be the Ornstein--Uhlenbeck semigroup on $\R$, $\mu$ the Gaussian measure on $\mathbb R$, $\mu(\ud x)=\frac{1}{\sqrt{\pi}}\e^{-|x|^2}\ud x$, and $S_m(\cdot)$ be the Ornstein--Uhlenbeck semigroup on ${\mathcal S}_m$. We know $\mu$ is the invariant measure of $S(\cdot)$, hence
$$ \int_\R S(t) \sum_{i=1}^m\tilde{f}_i(x)\,\mu (\ud x)=\int_\R \sum_{i=1}^m\tilde{f}_i(x)\,\mu (\ud x) \quad \hbox{for all } t >0,$$
and, because $\sum_{i=1}^m\tilde{f}_i(x)$ and therewith $S(t) \sum_{i=1}^m\tilde{f}_i(x)$ are even functions on $\R$, we infer that $\frac{1}{\sqrt{\pi}}\e^{-|x|^2} \ud x$, $i=1,\ldots,m$, defines an invariant measure for $S_m(\cdot)$. After normalizing this measure we may conclude that $\mu_m$ is indeed an invariant probability measure for $S_m(\cdot)$.

Uniqueness follows along the same lines as usual from the ergodicity of the invariant measure.
\end{proof}

\begin{rem}\label{rem:Lp}
As in Remark \ref{rem:Cb}, the regularity properties of the classical Ornstein--Uhlenbeck semigroup on $\R$, cf. \cite[Sections 9.3]{LorBer07}, hold for the semigroup $S_m(\cdot)$ on $L^p_{\mu_m}({\mathcal S}_m)$:
\begin{itemize}
\item For any $p\in (1,\infty)$, $S_m(\cdot)$ is analytic in $L^p_{\mu_m}({\mathcal S}_m)$ and consistent, i.e. $S_m(t)$ on $L^p_{\mu_m}({\mathcal S}_m)$ and on $L^q_{\mu_m}({\mathcal S}_m)$ coincide on $L^p_{\mu_m}({\mathcal S}_m)\cap L^q_{\mu_m}({\mathcal S}_m)$ for all $p,q\in (1,\infty)$ and $t\ge 0$.
\item  For any $p\in (1,\infty)$, $W^{1,p}_{\mu_m}({\mathcal S}_m)$ is compactly embedded in $L^p_{\mu_m}({\mathcal S}_m)$.
\item  The semigroup $S_m(t)$ maps $L^p_{\mu_m}({\mathcal S}_m)$ into $W^{1,p}_{\mu_m}({\mathcal S}_m)$ and hence $S_m(t)$ is compact in $L^p_{\mu_m}({\mathcal S}_m),\,1<p<\infty ,$ for any $t>0$.
\end{itemize}
\end{rem} 

For later purposes we propose to characterize the domain of the realization, $A_2$
of $A$ in $L^2_{\mu_m}({\mathcal S}_m)$. Here we recall that
\begin{eqnarray*}
& & L^2_{\mu_m}({\mathcal S}_m)=\bigoplus_{i=1}^mL^2_{\mu_m}(\R_+,i)\,\hbox{\ endowed with the norm } \\
& & \|f\|_{L^2_{\mu_m}({\mathcal S}_m)}^2=\sum_{i=1}^m\|f_i\|_{L^2_{\mu_m}(\R_+,i)}^2\,\hbox{\ for }f=(f_i)_{1\le i\le m}.
\end{eqnarray*}
Like in~\eqref{eq:sobol-defin1}, 
we define the weighted Sobolev spaces
$$\widetilde{H^k_{\mu_m}}({\mathcal S}_m):=\bigoplus_{i=1}^mH^k_{\mu_m}(\R_+,i),\quad k\in \mathbb{N}.$$
On $$D(a):=H^1_{\mu_m}({\mathcal S}_m):=\{f\in \widetilde{H^1_{\mu_m}}({\mathcal S}_m): f_i(0)=f_j(0) \hbox{\ for all }i,j=1,\ldots, m\}$$
we define the sesquilinear form 
$$a(f,g):=\frac{1}{2}\sum_{i=1}^m\int_{(\R_+,i)}f'_i(x_i)\overline{g'_i(x_i)}\mu_m(\ud x_i),\quad f,g\in D(a).$$
Since $a$ is densely defined, symmetric, accretive, continuous and closed sesquilinear form, we can associate the self-adjoint operator
\[
\begin{split}
D(B)&:=\left\{f\in D(a): \exists g\in L^2_{\mu_m}({\mathcal S}_m) \hbox{\ s.t. }a(f,\phi)=\langle g,\phi\rangle_{L^2_{\mu_m}({\mathcal S}_m)} \,\forall \phi\in D(a)\right\},\\
Bf&:=g.
\end{split}
\]
We can finally describe the  realization $A_2$ of the Ornstein--Uhlenbeck operator in $L^2_{\mu_m}({\mathcal S}_m)$.
\begin{prop}\label{domain-A2}
The generator $A_2$ of the Ornstein--Uhlenbeck semigroup on $L^2_{\mu_m}({\mathcal S}_m)$ is given by
\[
\begin{aligned}
D(A_2) &=\left\{f\in \widetilde{H^2_{\mu_m}}({\mathcal S}_m): f_i(0)=f_j(0) \hbox{\ for all }i,j=1,\ldots, m \hbox{\ and }\sum_{i=1}^mf'_i(0)=0\right\}\\
(A_2f)_i(x) &= \frac{1}{2}f''_i(x)-xf'_i(x),\quad\hbox{for all } f=(f_i)_{1\le i\le m}\in D(A_2).
\end{aligned}
\]
\end{prop}
\begin{proof}
Let $f\in D:=\{f\in \widetilde{H^2_{\mu_m}}({\mathcal S}_m): f_i(0)=f_j(0) \hbox{\ for all }i,j=1,\ldots, m \hbox{\ and }\sum_{i=1}^mf'_i(0)=0\}.$ Then $f\in D(a)$ and integrating by part one obtains
$$\langle -A_2f,\phi\rangle_{L^2_{\mu_m}({\mathcal S}_m)} =a(f,\phi),\quad \forall \phi\in D(a).$$
So, $(-A_2,D)\subseteq (B,D(B))$. 

Now, let $f=(f_i)_{1\le i\le m}\in D(B)$. Then, there is $g\in L^2_{\mu_m}({\mathcal S}_m)$ such that
\begin{equation}\label{eq:form-domain}
a(f,\phi)=\langle g,\phi\rangle_{L^2_{\mu_m}({\mathcal S}_m)},\quad \forall \phi\in D(a).
\end{equation}
For any fixed $j\in \{1,\ldots m\}$ consider the function $\phi =(\phi_i)_{1\le i\le m}$ with $\phi_j\in C_c^\infty(\R_+,i)$ and $\phi_i\equiv 0$ for $i\neq j$. Applying \eqref{eq:form-domain} with $\phi$ as above, one can see that $f_j\in H^2_{\mu_m}(\R_+,j)$ and $-\frac{1}{2}f''_j+x_jf'_j=g_j$. Thus, $f\in \widetilde{H^2_{\mu_m}}({\mathcal S}_m)$. So, we can integrate by part in \eqref{eq:form-domain} and obtain, for any $\phi\in D(a)$,
\begin{eqnarray*}
a(f,\phi) &=& \frac{1}{2}\sum_{i=1}^m\int_{(\R_+,i)}f'_i(x_i)\overline{\phi '_i(x_i)}\mu_m(\ud x_i)\\
&=& \sum_{i=1}^m \int_{(\R_+,i)}\left(-\frac{1}{2}f''_i(x_i)+x_if'_i(x_i)\right))\overline{\phi_i(x_i)}\mu_m(\ud x_i)+\frac{1}{2}\sum_{i=1}^mf'_i(0)\overline{\phi_i(0)}\\
&=& \langle g,\phi\rangle_{L^2_{\mu_m}({\mathcal S}_m)}+\frac{1}{2}\sum_{i=1}^mf'_i(0)\overline{\phi_i(0)}.
\end{eqnarray*}
By choosing now $\phi \in D(a)$ such that $\phi(0)\neq 0$, one obtains $f\in D$. Hence, $(-A_2,D)=(B,D(B))$.
\end{proof}
Before characterizing the spectrum of $A$., we need a preparatory lemma. The following seems to be folklore, but we could not find an appropriate reference in the literature. Because $\R\simeq \mathcal S_2$, with a slight abuse of notation we still denote by $A_2$ the realization of the Ornstein--Uhlenbeck operator on $L^2_{\mu}(\R)\simeq L^2_{\mu_2}(\mathcal S_2)$.
\begin{lemma}\label{lem:spectrumaevenodd}
The realization of the Ornstein--Uhlenbeck operator on $L^2_{\mu_1}(\R_+)$
has purely point spectrum given by
\[
\begin{cases}
\{-2k:k\in \N_0\},\quad &\hbox{if Neumann conditions are imposed at 0},\\
\{-2k-1:k\in \N_0\},\quad &\hbox{if Dirichlet conditions are imposed at 0}.\\
\end{cases}
\]
\end{lemma}

\begin{proof}
It is well known~\cite[Theorem~3.1]{MetPalPri02} that the spectrum of
the Ornstein--Uhlenbeck operator on $L^2_\mu (\R)$ consists precisely of the simple eigenvalues $k=0,-1,-2,\ldots$,  and that the corresponding eigenfunctions are given by the Hermite polynomials $H_k$, where
$$H_k(x):=(-1)^k\e^{|x|^2}D^k\e^{-|x|^2},\quad x\in \R,\ k\in \N_0.$$

This information can be reformulated: since we know that $A_2$ leaves invariant the mutually orthogonal subspaces $L^2_{\hbox{\scriptsize odd}}$ and $L^2_{\hbox{\scriptsize even}}$ (of odd and even $L^2_\mu(\R)$-functions, respectively), the spectrum of $A$ can be described as the disjoint union of two subsets: the spectrum of the restrictions of $A$ to $L^2_{\hbox{\scriptsize odd}}$ and $L^2_{\hbox{\scriptsize even}}$. In turn, these restrictions are unitarily equivalent (and isospectral) with the realizations $A_D$ and $A_N$ of the Ornstein--Uhlenbeck operator on $L^2_{\mu_1}(\R_+)$ with Dirichlet and Neumann conditions, respectively.

Since $H_k$ is a polynomial, $AH_k=-kH_k$; furthermore, $H'_k(0)=0$ if and only if is even, whereas $H_k(0)=0$ if and only if $k$ is odd. It follows that $H_k$ is an eigenfunction of $A_N$ whenever $k$ is even and $H_k$ is an eigenfunction of $A_D$ whenever $k$ is odd. This yields the claim.
\end{proof}

We now characterize the spectrum of $A_p$.
\begin{theo}\label{spectrum-A2}
The spectrum of the realization, $A_p$, $p \in (1,\infty)$, of $A$ in $L^p_{\mu_m}({\mathcal S}_m)$ consists of isolated eigenvalues and is independent of $p\in (1,\infty)$. Moreover, 
\[
 \sigma(A_p)=\{-k: k\in \N_0\},\qquad p\in (1,\infty),
 \]
 where all even eigenvalues have multiplicity 1, whereas all odd eigenvalues have multiplicity $m-1$.
\end{theo}

\begin{proof}
By Remark \ref{rem:Lp}, we know that $S_m(t)$ is compact in $L^p_{\mu_m}({\mathcal S}_m)$ for any $t>0$. Hence, the spectrum $\sigma(A_p)$ of $A_p$ consists of a sequence of eigenvalues. By standard arguments, see the proof of \cite[Proposition 2.10]{LorRha15},
one deduces that $\sigma(A_p)$ is independent of $p$, cf.~\cite[Section~7.2.2]{Are04} . Anyway for the reader’s convenience we give some details. From Remark \ref{rem:Lp}, we know that $R(\lambda , A_p) = R(\lambda , A_q )$ on $L^p_{\mu_m}({\mathcal S}_m)\cap L^q_{\mu_m}({\mathcal S}_m)$ for any $\lambda > 0$. Since $\sigma(A_p)$ and $\sigma(A_q )$ consist of isolated eigenvalues,
$\mathbb{C} \setminus (\sigma (A_p) \cup \sigma(A_q ))$ is a connected open set in $\mathbb{C}$. Hence, $R(\lambda , A_p) = R(\lambda , A_q )$ on
$L^p_{\mu_m}({\mathcal S}_m)\cap L^q_{\mu_m}({\mathcal S}_m)$ for any $\lambda \in \mathbb{C} \setminus (\sigma (A_p) \cup \sigma(A_q ))$.

Let us now fix $\lambda_0\in \sigma(A_p)$. So, $\lambda_0$ is isolated in $\sigma(A_p)\cup \sigma(A_q)$. Thus, there is $\varepsilon >0$ small enough such that $B_\varepsilon (\lambda_0)\setminus \{\lambda_0\}\subset \mathbb{C} \setminus (\sigma (A_p) \cup \sigma(A_q ))$. Let $P$ be the spectral projection associated with $\lambda_0$, which is defined by
$$Pf=\frac{1}{2\pi i}\int_{\partial B_\varepsilon(\lambda_0)^+}R(\lambda , A_p)f \ud x,\quad f\in L^p_{\mu_m}({\mathcal S}_m).$$
If $\lambda_0\not\in \sigma(A_q)$, then we have
$$P\varphi =\frac{1}{2\pi i}\int_{\partial B_\varepsilon(\lambda_0)^+}R(\lambda , A_q)\varphi \ud x=0$$
for all $\varphi \in C_c^\infty ({\mathcal S}_m)$. Thus, by density, $P\equiv 0$, which is a contradiction. Therefore, $\sigma(A_p)\subset \sigma(A_q)$ and hence $\sigma(A_p)= \sigma(A_q)$, since $p$ and $q$ have been arbitrarily fixed. In particular we have $\sigma(A_p)=\sigma(A_2)$.

Let us now turn to the task of describing the spectrum of $A_2$ on $L^2_{\mu_m}(\mathcal S_m)$. We will adapt a method which the first-named author has learned from Pavel Kurasov: it was, e.g., already used in~\cite[Section~3.5]{Mal13} to solve the 
problem of determining the spectrum of the Laplacian with natural vertex conditions on equilateral star graphs.

To begin with, we observe that $A_2$ leaves invariant the the mutually orthogonal subspaces $L^2_{\hbox{\scriptsize odd}}$ and $L^2_{\hbox{\scriptsize even}}$ of odd and even $L^2_{\mu_m}({\mathcal S}_m)$-functions, respectively\footnote{ By a slight abuse of notation, we are thereby calling a function $f:{\mathcal S}_m\to \R$ 
\begin{itemize}
\item \textit{even}, if $f(x_i)=f(x_j)$ for all $i,j=1,\ldots,m$,
\item \textit{odd}, if $f(x_1)+\ldots+f(x_m)=0$.
\end{itemize}
}:  up to minor modifications, this can be proved as in~\cite[Proposition~6.88]{Mug14}.
In fact, more is true: if we denote by $R$ the bounded, unitary operator on $L^2_{\mu_m}(\mathcal S_m)$ defined by
\[
R:(f_1,\ldots,f_{m-1},f_m)\mapsto (f_2,\ldots,f_m,f_1),
\]
then it is easy to see that $R$ commutes with $A_2$. Accordingly, by the Spectral Theorem for normal operators $A_2$ and $R$ can be simultaneously diagonalized: i.e., any eigenfunction of $A_2$ turns out to be an eigenfunction of $R$, and vice versa. So, what are the eigenfunctions of $R$? Observe that $R^m$ is the identity operator of $L^2_{\mu_m}(\mathcal S_m)$, so its eigenvalues are precisely the $m$-th roots of unity: $\e^{\frac{2j\pi i}{m}}$, $j=0,\ldots,m-1$. A direct computation shows that the corresponding (infinite-dimensional) $j$-th eigenspace of $R$ is
\[
E_j:=(1,z^j,z^{2j},\ldots,z^{j(m-1)})\otimes L^2_{\mu_m}(\R_+),
\]
where $z:=\e^{\frac{2\pi i}{m}}$. Observe that the eigenspace $E_0$ agrees with the space $L^2_{\hbox{\scriptsize even}}$ of even $L^2_{\mu_m}(\mathcal S_m)$-functions, whereas the remaining $m-1$ eigenspaces of $R$ consist of odd functions: we conclude that $L^2_{\hbox{\scriptsize odd}}=\bigoplus_{j=1}^{m-1}E_j$.

Let us first study the restriction of $A_2$ to $L^2_{\hbox{\scriptsize odd}}$, or equivalently $\bigoplus_{j=1}^{m-1} A_{|E_j}$:  for continuous functions on $\mathcal S_m$ (like those in $D(A_2)$), oddness induces, by
\[
0=\sum_{i=1}^m f(0_i)=mf(0),
\]
 Dirichlet boundary conditions at 0: we conclude that the spectrum of $A_{|E_j}$ is 
 \[
 \{-2k-1:k\in\N_0\}\quad\hbox{ for all }j=1,\ldots,m-1.
 \]

Likewise, the restriction of $A_2$ to the space $E_0$ of even functions on $\mathcal S_m$ is isomorphically equivalent, hence isospectral, with the realization $A_N$ of the Ornstein--Uhlenbeck operator on $L^2_{\mu_1}(\R_+)$ with Neumann conditions at 0: we already know from Lemma~\ref{lem:spectrumaevenodd}
that the corresponding eigenvalues form the set
\[
\{-2k:k\in \N_0\}.
\]
This concludes the proof.
\end{proof}

We have just seen that the Ornstein--Uhlenbeck semigroup generated by $A_2$ in $L^2_{\mu_m}(\mathcal S_m)$ is compact. In fact, more can be said.
\begin{prop}\label{prop:trace}
The Ornstein--Uhlenbeck semigroup generated by $A_2$ on $L^2_{\mu_m}(\mathcal S_m)$ is of trace class.
\end{prop}
\begin{proof}
It suffices to observe that the $L^2_{\mu_m}(\mathcal S_m)$-eigenvalues of $(\lambda -A_2)^{-1},\,\lambda >0,$ are square summable for all $m\in \N$. Accordingly, $A_2$ has Hilbert--Schmidt resolvent, and the trace class property of the semigroup follows.
\end{proof}
In the case of $m=2$, the assertion in Proposition~\ref{prop:trace} is a direct consequence of \cite[Theorem~3.3]{MetPal00}.


\section{The harmonic oscillator}\label{sec:harmonic}

Let us discuss a further example: the harmonic oscillator 
\begin{equation}\label{eq:opou-ho}
B f(x_i)= \frac12 \left(f''(x_i) -|x_i|^2 f(x_i)+f(x_i)\right), \qquad |x_i|\ge 0, \quad  i=1,\ldots, m, 
\end{equation}
on the star graph ${\mathcal S}_m$, again with Kirchhoff-type conditions in $0$, i.e.,
\begin{equation}\label{eq:domou-ho}
D(B)=\{f \in C_b({\mathcal S}_m) \cap  \bigcap_{1\leq p < \infty} \widetilde{W^{2,p}_{\hbox{\scriptsize loc}}}({\mathcal S}_m):  \sum_{i=1}^m f'(0_i)=0 \hbox{ and } Bf \in C_b({\mathcal S}_m)\}. 
\end{equation}
For $m=2$ we have the classical harmonic oscillator on $\R$: we refer to~\cite[Section~4.3]{Dav89} for basic facts about it. In particular, it is known that $B$ generates on $L^2(\R)$ an ultracontractive semigroup whose heat kernel is given by
\begin{equation}\label{eq:Mehler-HO}
k(t,x,y):=\frac{1}{\sqrt{\pi(1-\e^{-2 t})}}\e^{\frac{4xy\e^{-t}-(x^2+y^2)(1+\e^{-2t})}{2(1-\e^{-2t})}}
\end{equation}
by the celebrated Mehler formula.	

Applying Theorem \ref{main-thm} with $q(x)=\frac{1}{2},\,b(x)=0$ and $c(x)=-\frac{1}{2}(x^2-1)$ and \eqref{eq:Mehler-HO} we deduce the following.
\begin{cor}
For every function $f \in C_b({\mathcal S}_m)$, there exists a unique classical solution $u$ of $({\rm P}_B)$ given by the semigroup $U_m(\cdot)$
\begin{equation}\label{semiU} 
\begin{split}
u(t,x_i)&= (U_m(t)f)(x_i)\\
&= \frac{2}{\sqrt{\pi(1-\e^{-2t})}}\int_{(\R_+,i)} \Big ( \exp\left[-\frac{(1+\e^{-2t})(|x_i|^2+|y_i|^2)}{2(1-e^{-2t})}\right] \sinh\left(\frac{2|x_i||y_i|e^{-t}}{1-e^{-2t}}\right) f(y_i)\,\ud y_i \\
&\qquad  + \frac{2}{m\sqrt{\pi(1-\e^{-2t})}}\sum_{\substack{1\le j\le m}}
 \int_{(\R_+,j)} \e^{\frac{-4|x_i|y_j|\e^{-t}-(|x_i|^2+|y_j|^2)(1+\e^{-2t})}{2(1-\e^{-2t})}}f(y_j)\,\ud y_j
\end{split}
\end{equation}
for $1\le i\le m$.
\end{cor}

As in the previous section we describe the realization $B_2$ of the harmonic oscillator $B$ in 
\begin{eqnarray*}
& & L^2({\mathcal S}_m)=\bigoplus_{i=1}^mL^2(\R_+,i)\,\hbox{\ endowed with the norm } \\
& & \|f\|_{L^2({\mathcal S}_m)}^2=\sum_{i=1}^m\|f_i\|_{L^2(\R_+,i)}^2\,\hbox{\ for }f=(f_i)_{1\le i\le m}.
\end{eqnarray*}
To this purpose we consider the isometry
\begin{eqnarray*}
& & T : L^2_{\mu_m}({\mathcal S}_m)\to L^2({\mathcal S}_m)\\
& &\quad \quad \quad \quad \quad f\mapsto (\sqrt{c_m}\e^{-\frac{x^2}{2}}f_i),
\end{eqnarray*}
where $c_m:=\frac{2}{m\sqrt{\pi}}$. An easy computation shows that
$B=TAT^{-1}$ and so by Proposition \ref{domain-A2} and Theorem \ref{spectrum-A2} we have the following result.
\begin{prop}\label{domain-B2}
The generator $B_2$ of the harmonic oscillator semigroup $(U_m(\cdot)$ on $L^2({\mathcal S}_m)$ is given by
\[
\begin{aligned}
D(B_2) &=\left\{f\in \widetilde{H^2}({\mathcal S}_m): f_i(0)=f_j(0) \hbox{\ for all }i,j=1,\ldots, m \hbox{\ and }\sum_{i=1}^mf'_i(0)=0\right\}\\
(B_2f)_i(x) &= \frac{1}{2}(f''_i(x)-x^2f_i(x)+f_i(x)),\quad\hbox{for all } f=(f_i)_{1\le i\le m}\in D(B_2).
\end{aligned}
\]
Moreover, $U_m(\cdot)=TS_m(\cdot)T^{-1}$ and 
\[
 \sigma(B_2)=\{-k: k\in \N_0\},
 \]
 where all even eigenvalues have multiplicity 1, whereas all odd eigenvalues have multiplicity $m-1$.
\end{prop}

\section*{Acknowledgements}
The authors would like to thank Marvin Plümer (Hagen) for interesting discussions.

\end{document}